%% file: improved_rate.tex
\documentclass[review,onefignum,onetabnum]{siamonline190516}

\usepackage{stmaryrd}
\usepackage{dsfont}
\usepackage{enumerate}
\usepackage{graphicx,color}
\usepackage{amsmath}
\usepackage{amsfonts}
\usepackage{amssymb}
\usepackage{multirow}

\newsiamthm{claim}{Claim}
\newsiamremark{remark}{Remark}
\newsiamremark{hypothesis}{Hypothesis}
\crefname{hypothesis}{Hypothesis}{Hypotheses}

\Crefname{ALC@unique}{Line}{Lines}

\usepackage{amsopn}

\usepackage{xspace}
\usepackage{bold-extra}
\usepackage[most]{tcolorbox}

\colorlet{texcscolor}{blue!50!black}
\colorlet{texemcolor}{red!70!black}
\colorlet{texpreamble}{red!70!black}
\colorlet{codebackground}{black!25!white!25}

\lstdefinestyle{siamlatex}{%
	style=tcblatex,
	texcsstyle=*\color{texcscolor},
	texcsstyle=[2]\color{texemcolor},
	keywordstyle=[2]\color{texemcolor},
	moretexcs={cref,Cref,maketitle,mathcal,text,headers,email,url},
}

\tcbset{%
	colframe=black!75!white!75,
	coltitle=white,
	colback=codebackground, 
	colbacklower=white, 
	fonttitle=\bfseries,
	arc=0pt,outer arc=0pt,
	top=1pt,bottom=1pt,left=1mm,right=1mm,middle=1mm,boxsep=1mm,
	leftrule=0.3mm,rightrule=0.3mm,toprule=0.3mm,bottomrule=0.3mm,
	listing options={style=siamlatex}
}

\newtcblisting[use counter=example]{example}[2][]{%
	title={Example~\thetcbcounter: #2},#1}

\newtcbinputlisting[use counter=example]{\examplefile}[3][]{%
	title={Example~\thetcbcounter: #2},listing file={#3},#1}

\DeclareTotalTCBox{\code}{ v O{} }
{ 
	fontupper=\ttfamily\color{black},
	nobeforeafter,
	tcbox raise base,
	colback=codebackground,colframe=white,
	top=0pt,bottom=0pt,left=0mm,right=0mm,
	leftrule=0pt,rightrule=0pt,toprule=0mm,bottomrule=0mm,
	boxsep=0.5mm,
	#2}{#1}

\patchcmd\newpage{\vfil}{}{}{}
\begin{tcbverbatimwrite}{tmp_\jobname_header.tex}
	\title{On the Improved Rates of Convergence for Mat\'ern-type Kernel Ridge Regression, with Application to Calibration of Computer Models\thanks{Submitted to the editors DATE.
			\funding{Tuo's work is supported by NSF grants DMS-1914636 and DMS-1564438, and also by the National Center for Mathematics and Interdisciplinary Sciences in CAS and NSFC grants 11501551, 11271355 and 11671386. Wu's work is supported by NSF grants DMS-1564438 and DMS-1914632.}}}
	
	\author{Rui Tuo\thanks{Department of Industrial and Systems Engineering, Texas A\&M University,	College Station, TX 77843, USA (\email{ruituo@tamu.edu}).}
		\and Yan Wang\thanks{College of Appied Sciences, Beijing University of Technology, Beijing 100124, China (\email{yanwang@bjut.edu.cn}).}
		\and C. F. Jeff Wu\thanks{School of Industrial and Systems Engineering,	Georgia Institute of Technology, Atlanta, Georgia 30332, USA (\email{jeff.wu@isye.gatech.edu})}}
	
	\headers{Improved Rates for Kernel Ridge Regression}{Rui Tuo, Yan Wang and C. F. Jeff Wu}
\end{tcbverbatimwrite}
\input{tmp_\jobname_header.tex}

\ifpdf
\hypersetup{ pdftitle={Guide to Using  SIAM'S \LaTeX\ Style} }
\fi

\begin{document}

\maketitle

\begin{abstract}
Kernel ridge regression is an important nonparametric method for estimating smooth functions. We introduce a new set of conditions, under which the actual rates of convergence of the kernel ridge regression estimator under both the $L_2$ norm and the norm of the reproducing kernel Hilbert space exceed the standard minimax rates. An application of this theory leads to a new understanding of the Kennedy-O'Hagan approach [J. R. Stat. Soc. Ser. B. Stat. Methodol. 63 (2001) 425–464] for calibrating model parameters of computer simulation. We prove that, under certain conditions, the Kennedy-O'Hagan calibration estimator with a known covariance function converges to the minimizer of the norm of the residual function in the reproducing kernel Hilbert space.
\end{abstract}

\begin{keywords}
{nonparametric regression},
{reproducing kernel Hilbert space},
{kriging},
{calibration of parameters}
\end{keywords}

\begin{AMS}
  62G08, 62M30, 62M40
\end{AMS}

\section{Introduction}

A major challenge in computer simulation of complex systems is to choose suitable model parameters. These parameters usually represent specific intrinsic attributes of the system. The input values of the model parameters can significantly affect the accuracy and  usefulness of the computer output. When physical observations are available, one can adjust the computer model parameters so that the computer outputs match the physical data. We call this activity the calibration of computer models.

The celebrated Bayesian calibration method by Kennedy and O'Hagan \cite{kennedy2001bayesian} is one of the major and widely used approaches for the calibration of computer models. A remarkable contribution of \cite{kennedy2001bayesian} is to incorporate a ``discrepancy function'' to model the difference between the computer outputs and the physical process. This discrepancy does exist in most computer simulation problems, because we have to resort to simplifications and unrealistic assumptions when building the computer models.

Without an informative prior, the Kennedy and O'Hagan model is non-identifiable, because one cannot determine the model parameters and the discrepancy function simultaneously. Kennedy and O'Hagan \cite{kennedy2001bayesian} used a Gaussian process as a prior for the discrepancy function. Tuo and Wu \cite{tuo2016theoretical} conducted a theoretical study on a simplified version of the Kennedy-O'Hagan method (abbreviated as the K-O method) when the physical data are noiseless. Under this condition, the radial basis functions approximation can be regarded as a frequentist version of Gaussian process regression. With the help of related mathematical tools, Tuo and Wu \cite{tuo2016theoretical} identified the limit value of the Kennedy-O'Hagan method as well as the rate of convergence.

A primary goal of this work is to establish an asymptotic theory for the K-O method with noisy physical data. The frequentist version of the Gaussian process regression, in this situation, is the kernel ridge regression \cite{saunders1998ridge}. 
With an improved rate of convergence for kernel ridge regression, we prove that, under certain conditions, the K-O estimator tends to the parameter value which minimizes the norm of the residual function in the reproducing kernel Hilbert space. We also present the rate of convergence of the K-O estimator. 
As a consequence, we relax a key and rather restrictive assumption in \cite{tuo2016theoretical}. 
Tuo and Wu had to assume that the physical experiments have no random errors, which is not realistic.

There is a vast literature on the theoretical properties of ridge kernel regression. It is known that the rate of convergence of this method can be improved by imposing extra smoothness conditions on the underlying function; see, e.g., \cite{gu2002smoothing}. We refer to \cite{blanchard2018optimal,dicker2017kernel,guo2017learning,lin2017distributed} and the references therein for the recent advances in this area. In this work, we present some results on the improved rates similar to the above works. Compared with the existing ones, our model settings are closer to the practical applications in engineering and computer experiments. First, the existing methods focus on kernels constructed by a set of eigenvalues and orthonormal basis functions. This construction has, albeit mathematically general, not been widely used in practice, because the computational cost is high, and an orthonormal basis may be difficult to obtain for a general input domain. In this work, we consider the widely used Mat\'ern kernels. Second, the existing results focus on random designs, which are not usually adopted in engineering. The present work considers fixed designs satisfying some space-filling properties. Third, the existing results for the improved rates are not sufficient to develop an asymptotic theory for the K-O method. We obtain a strengthened version of the improved rates, which lead to the desired asymptotic theory for calibration. It is worth noting that the mathematical treatments in this paper differ from those in the works mentioned above, and our work provides some new insight on kernel ridge regression.

This article is organized as follows. In Section \ref{sec:irkrr}, we introduce some background of this work and present the improved rates of convergence for kernel ridge regression. In Section \ref{sec:gaussian}, we establish an asymptotic theory for the K-O calibration estimators. In Section \ref{sec:simulation}, we validate our theoretical assertion with a numerical study. Concluding remarks are made in Section \ref{sec:discussion}. Appendix \ref{App:proof} contains the long proofs in this article.

\section{Improved rates for kernel ridge regression}\label{sec:irkrr}
In this section, we discuss the mathematical tool and our results on the improved rates of convergence for kernel ridge regression.
\subsection{Overview}
Consider a nonparametric regression model
\begin{eqnarray}\label{NP}
y_i=f(x_i)+e_i,
\end{eqnarray}
where $f$ is a smooth function whose domain of definition $\Omega$ is a convex and compact subset of $\mathbb{R}^d$, and $e_i$'s are independent and identically distributed random sequence with mean zero and finite variance. The problem of interest is to recover $f$ from the data $(x_i,y_i), i=1,\ldots,n$.
Kernel ridge regression is one of the important methods to deal with this problem. This method has been widely used in statistics and machine learning \cite{saunders1998ridge}. It also has close relationships with classic kernel-based regression methods like smoothing splines or thin-plate splines \cite{wahba1990spline}.

Suppose $f$ lies in the Sobolev space $H^m(\Omega)$ with $m>d/2$. By choosing a kernel function with $m$ degree of smoothness, the kernel ridge regression, as defined in (\ref{eq2.3}), can reach the standard rates of convergence
\begin{eqnarray}
\|\hat{f}_n-f\|_{L_2(\Omega)}&=&O_p(n^{-\frac{m}{2m+d}}),\label{ratelow}\\
\|\hat{f}_n-f\|_{H^m(\Omega)}&=&O_p(1),\label{Op1}
\end{eqnarray}
where $\|\cdot\|_{L_2(\Omega)}$ and $\|\cdot\|_{H^m(\Omega)}$ denote the corresponding $L_2$ and Sobolev norm respectively. See, for example, \cite{gu2002smoothing,geer2000empirical} for details.
These rates are known to be the minimax rates in the current context \cite{stone1982optimal}. That is, these rates are in general not improvable.

From (\ref{ratelow}), we can see that the convergence rate depends on the smoothness of the underlying function. If we assume a higher smoothness condition for $f$, we can achieve a better rate by applying the kernel ridge regression with a kernel function as smooth as $f$. However, the smoothness of most practical underlying functions is unknown. Therefore, usually we cannot identify the optimal kernel functions. In practice, kernel functions with relatively low smoothness are frequently used. For instance, in spatial statistics and computer experiments, Mat\'ern kernels (see Section \ref{sec:RKHS} for the definition) with smoothness parameter 3/2 or 5/2 are widely used \cite{stein2012interpolation,santner2013design}. In this article, we show that if the underlying function $f$ is smoother than the kernel function, the rate of convergence of the kernel ridge regression may be improved.
%
Specifically, we identify a dense subset $S\subset H^m(\Omega)$ in such a way that if $f\in S$, we can reach the improved rates of convergence
\begin{eqnarray}
\|\hat{f}_n-f\|_{L_2(\Omega)}&=&O_p(n^{-\frac{2m}{4m+d}}),\\
\|\hat{f}_n-f\|_{H^m(\Omega)}&=&O_p(n^{-\frac{m}{4m+d}}).\label{op1}
\end{eqnarray}

Clearly, there is a substantial improvement from (\ref{Op1}) to (\ref{op1}), because (\ref{Op1}) does not entail convergence. We also prove an improved rate of convergence under the norm of the reproducing kernel Hilbert space generated by the kernel function, denoted by $\mathcal{N}$, as
\begin{eqnarray}
\|\hat{f}_n-f\|_{\mathcal{N}}=O_p(n^{-\frac{m}{4m+d}}).\label{rateRKHS}
\end{eqnarray}



\subsection{Reproducing Kernel Hilbert Spaces}\label{sec:RKHS}

Our study will employ the reproducing kernel Hilbert spaces (also called the native spaces) as the mathematical tool. Let $\Omega$ be a subset of $\mathbf{R}^d$.
Assume that $K:\Omega\times\Omega\rightarrow {\rm{R}}$ is a symmetric positive definite kernel. Define the linear space
\begin{eqnarray}\label{FPhi}
F_{K}(\Omega)=\left\{\sum_{i=1}^n\beta_i K(\cdot,{x}_i):\beta_i\in \mathbb{R},{x}_i\in \Omega,n\in\mathbb{N}\right\},
\end{eqnarray}
and equip this space with the bilinear form
\begin{eqnarray}
\left\langle\sum_{i=1}^n\beta_i K(\cdot,{x}_i),\sum_{j=1}^m\gamma_j K(\cdot, x'_j)\right\rangle_K:=\sum_{i=1}^n\sum_{j=1}^m\beta_i\gamma_j K({x}_i, x'_j).
\label{eq1.2}
\end{eqnarray}
Then the \emph{reproducing kernel Hilbert space} $\mathcal{N}_{K}(\Omega)$ generated by the kernel function $K$ is defined as the closure of $F_{K}(\Omega)$ under the inner product $\langle\cdot,\cdot\rangle_{K}$, and the norm of  $\mathcal{N}_{K}(\Omega)$  is $\| f\|_{\mathcal{N}_{K}(\Omega)}=\sqrt{\langle f,f\rangle_{\mathcal{N}_{K}(\Omega)}}$, where $\langle\cdot,\cdot\rangle_{\mathcal{N}_{K}(\Omega)}$ is induced by $\langle \cdot,\cdot\rangle_{K}$. More detail about reproducing kernel Hilbert space can be found in \cite{wahba1990spline,wendland2004scattered}.

In this work, we suppose the kernel function $K$ is stationary, i.e., $K(x,y)$ depends only on $x-y$. We denote $K(x,y)=:\Phi(x-y)$ and also denote the reproducing kernel Hilbert space $\mathcal{N}_{K}(\Omega)$ by $\mathcal{N}_{\Phi}(\Omega)$. Specifically, we focus on the Mat\'ern kernel function \cite{santner2013design,stein2012interpolation} defined by
\begin{eqnarray}\label{matern}
\Phi(x;\nu,\phi)=\frac{1}{\Gamma(\nu)2^{\nu-1}}(2\sqrt{\nu}\phi |x|)^\nu K_\nu(2\sqrt{\nu}\phi|x|),
\end{eqnarray}
where $K_\nu$ is the modified Bessel function of the second kind; $\nu$ and $\phi$ are \textit{fixed} parameters.
In (\ref{matern}), $\phi$ is a scale parameter, and $\nu$ is often called the smoothness parameter because it is related to the smoothness of the Gaussian processes associated with this kernel (covariance) function.

The smoothness of the kernel $\Phi$ is somehow inherited by the reproducing kernel Hilbert space $\mathcal{N}_{\Phi}(\Omega)$ \cite[Theorem 10.45]{wendland2004scattered}. Specifically, if $\Phi$ is a Mat\'ern kernel in (\ref{matern}), $\mathcal{N}_{\Phi}(\Omega)$ is equal to the (fractional) Sobolev space $H^{\nu+d/2}(\Omega)$, with equivalent norms. See also Corollary 1 of \cite{tuo2016theoretical}. Here we see that the smoothness parameter $\nu$ is also related to the smoothness of the Sobolev space.

\subsection{An Improved Rate in Scattered Data Approximation}\label{sec:SDA}

The current work is partially inspired by a result in scattered data approximation \cite{wendland2004scattered}, which gives an improved rate of convergence for radial basis function interpolation. In this section, we briefly review this result.

Let $f$ be an underlying function. Suppose we have observed the function values of $f$ over some scattered points $X=\{x_1,\ldots,x_n\}$. Then, an interpolant of $f$ is constructed by solving the optimization problem:
\begin{equation}
\begin{gathered}\label{RBF}
\min\|g\|_{\mathcal{N}_\Phi(\Omega)}\\
\text{s.t. } g(x_i)=f(x_i) \text{ for } i=1,\ldots,n.
\end{gathered}
\end{equation}
We denote this interpolant by $s_{f,X}$, which is commonly known as the radial basis function interpolant.
Formula (\ref{RBF}) is the limit case of the kernel ridge regression estimator introduced later in (\ref{eq2.3}), with $\lambda_n\downarrow 0$. 

The error estimate for radial basis function interpolant is well established in the literature. See \cite{wendland2004scattered}. Suppose $\mathcal{N}_\Phi(\Omega)$ is continuously embedded into a (fractional) Sobolev space $H^m(\Omega)$ and the design $X$ is quasi-uniform (see Section \ref{sec ir} for the formal definition). Then a standard error bound is
\begin{eqnarray}\label{ee1}
\|f-s_{f,X}\|_{L_2(\Omega)}\leq C n^{-d/m} \|f-s_{f,X}\|_{\mathcal{N}_\Phi(\Omega)},
\end{eqnarray}
for some constant $C$ independent of $f$, $n$ and the choice of a quasi-uniform design. Radial basis function interpolation satisfies the orthogonality condition
\begin{eqnarray}\label{orthogonality}
\langle f-s_{f,X},s_{f,X}\rangle_{\mathcal{N}_\Phi(\Omega)}=0,
\end{eqnarray}
which implies $\|f-s_{f,X}\|_{\mathcal{N}_\Phi(\Omega)}=O(1)$ as $n$ tends to infinity. Therefore, $\|f-s_{f,X}\|_{L_2(\Omega)}$ decays at least with the order $O(n^{-d/m})$ according to (\ref{ee1}).

To pursue an improved rate of convergence, one may ask whether $\|f-s_{f,X}\|_{\mathcal{N}_\Phi(\Omega)}=o(1)$. Although this does not hold generally \cite{Edmunds1996Function,santin2016approximation}, we do have an improved rate if there exists $v\in L_2(\Omega)$, so that
\begin{eqnarray}\label{T}
f(x)=\int_\Omega \Phi(x-t)v(t)d t.
\end{eqnarray}
Proposition 10.28 of \cite{wendland2004scattered} shows that functions with the form (\ref{T}) is a dense subset of $\mathcal{N}_\Phi(\Omega)$.
It shows that in this case for any $g\in\mathcal{N}_\Phi(\Omega)$,
\begin{eqnarray}\label{normequality}
\langle f,g\rangle_{\mathcal{N}_\Phi(\Omega)}=\langle v,g\rangle_{L_2(\Omega)}.
\end{eqnarray}
Combining (\ref{ee1}), (\ref{orthogonality}) and (\ref{normequality}), and applying Cauchy-Schwarz inequality yield
\begin{eqnarray}
\|f-s_{f,X}\|_{L_2(\Omega)}^2&\leq& C^2 n^{-2d/m} \|f-s_{f,X}\|_{\mathcal{N}_\Phi(\Omega)}^2\label{i1}\\
&=&C^2 n^{-2d/m} \langle f-s_{f,X}, f\rangle_{\mathcal{N}_\Phi(\Omega)}\nonumber\\
&=&C^2 n^{-2d/m} \langle f-s_{f,X}, v\rangle_{L_2(\Omega)}\nonumber\\
&\leq &C^2 n^{-2d/m} \|f-s_{f,X}\|_{L_2(\Omega)}\|v\|_{L_2(\Omega)}.\label{i2}
\end{eqnarray}
Canceling $\|f-s_{f,X}\|_{L_2(\Omega)}$ from both sides of (\ref{i2}) and comparing (\ref{i1}) and (\ref{i2}) yield the improved error bounds
\begin{eqnarray*}
	\|f-s_{f,X}\|_{L_2(\Omega)}&\leq& C^2 n^{-2d/m}\|v\|_{L_2(\Omega)},\\
	\|f-s_{f,X}\|_{\mathcal{N}_\Phi(\Omega)}&\leq& C n^{-d/m} \|v\|_{L_2(\Omega)}.
\end{eqnarray*}

In Section \ref{sec ir}, we will use the same assumption (\ref{T}) to derive improved rates of convergence for kernel ridge regression.

If a Mat\'ern kernel (\ref{matern}) is used, (\ref{T}) is equivalent to imposing certain higher-order smoothness condition. Before introducing the condition, we discuss the extension theorem of reproducing kernel Hilbert spaces.

\begin{proposition}\label{Prop:extension}
	Each $h\in\mathcal{N}_\Phi(\Omega)$ has an extension $h_e\in\mathcal{N}_\Phi(\mathbb{R}^d)$ which defines an isometric map from $\mathcal{N}_\Phi(\Omega)$ to $\mathcal{N}_\Phi(\mathbb{R}^d)$. In other words, $h_e|_\Omega=h$, and $\langle h_e,h'_e\rangle_{\mathcal{N}_\Phi(\mathbb{R}^d)}=\langle h,h'\rangle_{\mathcal{N}_\Phi(\Omega)}$ for all $h,h'\in\mathcal{N}_\Phi(\Omega)$, where $h_e|_\Omega$ denotes the restriction of $h_e$ on the region $\Omega$.
\end{proposition}

The main steps in proving Proposition \ref{Prop:extension} are as follows. First, we consider the map from $F_\Phi(\Omega)$ defined in (\ref{FPhi}) to $F_\Phi(\mathbb{R}^d)$ given by
\begin{eqnarray*}
	\sum_{i=1}^n \beta_i \Phi(x-x_i), x\in\Omega \mapsto \sum_{i=1}^n \beta_i \Phi(x-x_i), x\in\mathbb{R}^d,
\end{eqnarray*} 
which defines an extension for each function in $F_\Phi(\Omega)$.
Clearly, this map preserves the inner product (\ref{eq1.2}). Next, by using some functional analysis machinery such as taking Cauchy sequences, we can extend the domain of definition of this map from $F_\Phi(\Omega)$ to its closure, the Hilbert space $\mathcal{N}_\Phi(\Omega)$, and the extended map is also isometric.
We refer to Theorem 10.46 of \cite{wendland2004scattered} for details of the proof. 

Theorem \ref{th:existence} gives an equivalent statement of the condition (\ref{T}).

\begin{theorem}\label{th:existence}
	Suppose $\Phi$ is a Mat\'ern kernel (\ref{matern}) with smoothness parameter $\nu=m-d/2$, and $f\in\mathcal{N}_\Phi(\Omega)$.
	Then, the integral equation 
	\begin{eqnarray}\label{integralequation}
	f(x)=\int_\Omega \Phi(x-t) v(t) d t
	\end{eqnarray}
	has a solution $v\in L_2(\Omega)$ if and only if the extended function $f_e\in H^{2m}(\mathbb{R}^d)$. 
\end{theorem}

To maintain flow of the paper, all the long proofs are given in the appendix.

\begin{remark}
	Obviously, $f_e\in H^{2m}(\mathbb{R}^d)$ implies $f\in H^{2m}(\Omega)$. However, the converse is not necessarily true. 
	The stronger condition $f_e\in H^{2m}(\mathbb{R}^d)$ essentially requires the smoothness of the function across the boundary of $\Omega$. To illustrate this point, we consider a simple example. Suppose $\Omega=[-1,1], \nu=1/2, \phi=1$. Then the Mat\'ern kernel becomes $\Phi(x)=e^{-|x|}$. Let $f(x)=e^{1-x}$, $x\in[-1,1]$. Then $f\in C^\infty[-1,1]$. However, since $f(x)=\Phi(x-1)$ for $x\in[-1,1]$, according to the discussion after Proposition \ref{Prop:extension}, we have $f_e=\Phi(x-1)=e^{-|x-1|}$, which is not in $H^2(\mathbb{R})$.
\end{remark}


\subsection{Rates of Convergence for Kernel Ridge Regression}\label{sec ir}

In this section, we return to model (\ref{NP}). The goal is to estimate the underlying function $f$ from the data $\{(x_i,y_i)\}_{i=1}^n$. As in Section \ref{sec:RKHS}, we choose a positive definite kernel function $\Phi$.
The \textit{kernel ridge regression} estimator of $f$ is defined as
\begin{eqnarray}
\hat f_n=\operatorname*{argmin}_g \frac{1}{n}\sum_{i=1}^n(y_i-g(x_i))^2+\lambda_n \| g \|^2_{\mathcal{N}_{\Phi}(\Omega)},
\label{eq2.3}
\end{eqnarray}
where $\lambda_n>0$ is a tuning parameter to balance the bias and the variance.

The optimization problem (\ref{eq2.3}) can be solved analytically. With the help of the representer theorem \cite{wahba1990spline,scholkopf2001generalized}, we find that $\hat{f}$ has the form
\begin{eqnarray}\label{representer}
\hat{f}_n(x)=\sum_{i=1}^n c_i \Phi(x-x_i),
\end{eqnarray}
where $c_i$'s are undetermined coefficients. Substituting (\ref{representer}) into (\ref{eq2.3}) and invoking (\ref{eq1.2}), the estimation becomes a ridge regression problem weighted by the kernel matrix, and this is where the name ``kernel ridge regression'' comes from. After some calculations, we can find that the vector $c=(c_1,\ldots,c_n)^T$ is given by
\begin{eqnarray}\label{ridge}
c=(\mathbf{\Phi}+n\lambda_n  I_n)^{-1}Y,
\end{eqnarray}
where $\mathbf{\Phi}=(\Phi(x_i,x_j))_{i j}$, $Y=(y_1,\ldots,y_n)^T$ and $I_n$ is the identity matrix.

\subsubsection{Standard Rates of Convergence}

In this paper, we are interested in the conditions that ensure a consistent estimation for $f$ using the kernel ridge regression and the rate of convergence. First, we review the existing results and the standard proof.

Throughout the paper, we assume that the reproducing kernel Hilbert space $\mathcal{N}_\Phi(\Omega)$ is equal to some (fractional) Sobolev space $H^m(\Omega)$ with equivalent norms, for some $m>d/2$. Recall that if $\Phi$ is a Mat\'ern kernel in (\ref{matern}), $\mathcal{N}_\Phi(\Omega)$ is $H^{\nu+d/2}(\Omega)$.
We also assume that the random error $e_i$'s are sub-Gaussian in the sense that there exists universal constants $K,\sigma_0>0$ such that
\begin{eqnarray}
\mathbb{P}(|e_i|>t)\leq K e^{-t^2/\sigma_0^2}
\label{eq2.2}
\end{eqnarray}
holds for all $t>0$. This condition can be relaxed, but the technical details will become more involved and we do not pursue such a treatment here.

Define the \textit{empirical semi-norm} by
\begin{eqnarray*}
	\|f\|_n^2=\frac{1}{n}\sum_{i=1}^n f^2(x_i),
\end{eqnarray*}
and write $a\vee b=\max\{a,b\}$.
The standard convergence results are stated in Proposition \ref{th:existing}. 

\begin{proposition}\label{th:existing}
	Suppose $f\in H^m(\Omega)$ and $\lambda_n^{-1}=O(n^{\frac{2m}{2m+d}})$. Then the estimator $\hat{f}_n$ given by (\ref{eq2.3}) satisfies
	\begin{eqnarray}
	\begin{split}\label{EC1}	
	\|\hat{f}_n-f\|_n=O_p(\lambda_n^{1/2}\vee n^{-\frac{1}{2}}\lambda_n^{-\frac{d}{4m}}),\\
	\|\hat{f}_n\|_{\mathcal{N}_\Phi(\Omega)}=O_p(1 \vee n^{-\frac{1}{2}}\lambda_n^{-\frac{2m+d}{4m}}).
	\end{split}
	\end{eqnarray}
\end{proposition}

Because the main idea of proving Proposition \ref{th:existing} is also useful in establishing the improved rate of convergence, we give a sketch of proof for Proposition \ref{th:existing}. A detailed version can be found in Theorem 10.2 of \cite{geer2000empirical}.

\begin{proof}[Sketch of proof for Proposition \ref{th:existing}]
	The optimization condition (\ref{eq2.3}) implies the basic inequality
	\begin{eqnarray}\label{basicineq}
	\begin{split}
	&	\frac{1}{n}\sum_{i=1}^n(y_i-\hat{f}_n(x_i))^2+\lambda_n\|\hat{f}_n\|^2_{\mathcal{N}_\Phi(\Omega)}\\
	\leq& \frac{1}{n}\sum_{i=1}^n(y_i-f(x_i))^2+\lambda_n\|f\|^2_{\mathcal{N}_\Phi(\Omega)}.
	\end{split}
	\end{eqnarray}
	After some rearrangement, we can see that (\ref{basicineq}) is equivalent to
	\begin{eqnarray}\label{re1}
	\|\hat{f}_n-f\|^2_n+\lambda_n\|\hat{f}_n\|^2_{\mathcal{N}_\Phi(\Omega)}\leq 2\langle e,\hat{f}_n-f\rangle_n+\lambda_n\|f\|^2_{\mathcal{N}_\Phi(\Omega)},
	\end{eqnarray}
	where 
	\begin{eqnarray}\label{empiricalinnerproduct}
	\langle e,g\rangle_n:=\frac{1}{n}\sum_{i=1}^n e_i g_i(x_i).
	\end{eqnarray}
	It follows from a standard result in empirical process theory that
	\begin{eqnarray}\label{emp1}
	\frac{\langle e,\hat{f}_n-f\rangle_n}{\|\hat{f}_n-f\|^{1-\frac{d}{2m}}_n\|\hat{f}_n-f\|^{\frac{d}{2m}}_{\mathcal{N}_\Phi(\Omega)}}=O_p(n^{-1/2});
	\end{eqnarray}
	see Lemma \ref{lemma:empirical} in Appendix \ref{App:proof} for details. With some elementary algebraic calculations, also seeing Lemma \ref{lemma:order} and the proof of Theorem \ref{th:uniform} in the Appendix \ref{App:proof}, it is not hard to find that (\ref{re1}) and (\ref{emp1}) yield the desired results.
\end{proof}

The smoothing parameter $\lambda_n$ makes a tradeoff between the bias and the variance of the estimator. If $\lambda_n$ decays no faster than $O_p(n^{-\frac{2m}{2m+d}})$, the bias term dominates the variance term and the rate of convergence under the empirical semi-norm is $O_p(\lambda^{1/2}_n)$.
On the other hand, if $\lambda_n$ decays faster than $O_p(n^{-\frac{2m}{2m+d}})$, the variance term dominates the bias term and the rate of convergence under the $L_2$ norm is $O_p(n^{-\frac{1}{2}}\lambda_n^{-\frac{d}{4m}})$. In this case, $\|\hat{f}_n-f\|_{\mathcal{N}_\Phi(\Omega)}$ may go to infinity. Therefore, to reach the best rates of convergence, one needs to balance the bias and the variance. By choosing $\lambda_n\sim n^{-\frac{2m}{2m+d}}$, one can obtain the best rates
\begin{eqnarray*}
	\|\hat{f}_n-f\|_n&=&O_p(n^{-\frac{m}{2m+d}}),\\
	\|\hat{f}_n-f\|_{\mathcal{N}_\Phi(\Omega)}&=&O_p(1).
\end{eqnarray*}

An important question is whether the convergence results in (\ref{EC1}) imply a convergence under a more commonly used norm, like the $L_2$ norm. Such a result relies on whether the design points $\{x_1,\ldots,x_n\}$ are allocated in a space-filling manner. To address this point, we introduce the concept of quasi-uniformity \cite{brenner2007mathematical,utreras1988convergence}.

\begin{definition}
	For a set of design points $\{x_1,\ldots,x_n\}\subset\Omega$, define its \textit{fill distance} as
	\begin{eqnarray*}
		h_n=\max_{x\in\Omega}\min_{i}\|x-x_i\|,
	\end{eqnarray*}
	and its \textit{separation distance} as
	\begin{eqnarray*}
		q_n=\min_{i\neq j}\|x_j-x_i\|,
	\end{eqnarray*}
	where $\|\cdot\|$ denotes the Euclidean distance. Call a design sequence $x_1,\ldots,x_n,\ldots$ \textit{quasi-uniform}, if there exists a universal constant $B>0$ such that
	\begin{eqnarray}\label{quasiuniform}
	h_n/q_n\leq B
	\end{eqnarray}
	holds for all $n>1$.
\end{definition}

For any $\{x_1,\ldots,x_n\}\subset\Omega$, the balls centered at $x_i$'s with radius $q_n/2$ are disjoint. By comparing the volume of these balls and that of $\Omega$ we find that, the inequality
\begin{eqnarray}\label{qnbound}
n V_d (q_n/2)^d\leq 2 Vol(\Omega)
\end{eqnarray}
holds if $q_n$ is sufficiently small, where $V_d$ denotes the volume of $d$-dimensional unit ball, $Vol(\Omega)$ denotes the volume of $\Omega$.
If $\{x_1,\ldots,x_n\}$ also satisfies (\ref{quasiuniform}), (\ref{qnbound}) yields
\begin{eqnarray}\label{hbound}
h_n\leq \frac{B}{2}\left(\frac{2Vol(\Omega)}{V_d}\right)^{1/2}n^{-1/d}=:B'n^{-1/d}.
\end{eqnarray}

Under certain conditions, the empirical semi-norm and the $L_2$ norm are equivalent. The following Proposition is Lemma 3.4 of \cite{utreras1988convergence}.

\begin{proposition}\label{Prop:utreras}
	Suppose the design sequence is quasi-uniform.
	Then there exists a constant $C$ (depending only on $m$, $d$, $\Omega$ and $B$) and $h_0$ such that for any $ g \in H^m(\Omega)$ and $h_n\leq h_0$, we have
	\begin{eqnarray}
	\| g\|^2_{L_2(\Omega)}\leq C\left\{\| g\|^2_n+ h_n^{2m}\| g\|^2_{H^m(\Omega)}\right\}.
	\label{eq2.7}
	\end{eqnarray}
\end{proposition}

Corollary \ref{Coro:rate} gives the standard results for the rates of convergence of ridge kernel regression, which is a direct consequence of Proposition \ref{th:existing}, (\ref{hbound}) and Proposition \ref{Prop:utreras}.

\begin{corollary}\label{Coro:rate}
	Under the condition of Proposition \ref{th:existing}, suppose the design sequence is quasi-uniform. Then the estimator $\hat{f}_n$ given by (\ref{eq2.3}) satisfies
	\begin{eqnarray}
	\begin{split}\label{EC2}	
	\|\hat{f}_n-f\|_{L_2(\Omega)}=O_p(\lambda_n^{1/2}\vee n^{-\frac{1}{2}}\lambda_n^{-\frac{d}{4m}}),\\
	\|\hat{f}_n\|_{\mathcal{N}_\Phi(\Omega)}=O_p(1 \vee n^{-\frac{1}{2}}\lambda_n^{-\frac{2m+d}{4m}}).
	\end{split}
	\end{eqnarray}
\end{corollary}

\subsubsection{Improved Rates of Convergence}



We can regard the rates of convergence (\ref{EC2}) as a stochastic version of the error bound (\ref{ee1}). They are both standard convergence results under their respective settings. In view of the improved rate of convergence in interpolation discussed in Section \ref{sec:SDA}, we also expect an improved rate of convergence for the regression problem (\ref{NP}) by imposing the same assumption that there exists $v\in L_2(\Omega)$ so that (\ref{T}) holds.

Now we give more details about the intuition of why improved rates of convergence can be obtained. Note the identify
\begin{eqnarray}
\| f\|^2_{\mathcal{N}_{\Phi}(\Omega)}-\| \hat f_n\|^2_{\mathcal{N}_{\Phi}(\Omega)}
=2\langle  f, f-\hat f_n\rangle_{\mathcal{N}_{\Phi}(\Omega)}-\| f- \hat  f_n\|^2_{\mathcal{N}_{\Phi}(\Omega)},
\label{eq2.20}
\end{eqnarray}
which, together with the basic inequality (\ref{re1}), yields
\begin{eqnarray}\label{variation}
\begin{split}
&\|\hat{f}_n-f\|^2_n+\lambda_n\|\hat{f}_n-f\|^2_{\mathcal{N}_\Phi(\Omega)}\\
&\leq 2\langle e,\hat{f}_n-f\rangle_n+2\lambda_n\langle  f, f-\hat f_n\rangle_{\mathcal{N}_{\Phi}(\Omega)}.
\end{split}
\end{eqnarray}
Invoking identity (\ref{normequality}) and the Cauchy-Schwarz inequality, we obtain
\begin{equation*}
\langle  f, f-\hat f_n\rangle_{\mathcal{N}_{\Phi}(\Omega)}=\langle  v, f-\hat f_n\rangle_{L_2(\Omega)}\leq \|v\|_{L_2(\Omega)}\|\hat{f}_n-f\|_{L_2(\Omega)},
\end{equation*}
which, together with (\ref{variation}), implies
\begin{eqnarray}\label{improvedbasic}
\begin{split}
&\|\hat{f}_n-f\|^2_n+\lambda_n\|\hat{f}_n-f\|^2_{\mathcal{N}_\Phi(\Omega)}\\
&\leq 2\langle e,\hat{f}_n-f\rangle_n+2\lambda_n\|v\|_{L_2(\Omega)}\|\hat{f}_n-f\|_{L_2(\Omega)}.
\end{split}
\end{eqnarray}
We call (\ref{improvedbasic}) the \textit{improved basic inequality}, because it gives a refined version of the basic inequality (\ref{re1}). Compared to (\ref{re1}), the right-hand side of (\ref{improvedbasic}) is significantly deflated, because $\|\hat{f}_n\|^2_{\mathcal{N}_\Phi(\Omega)}$ in (\ref{re1}) has the order $O_p(1)$ according to Proposition \ref{th:existing}, while in (\ref{improvedbasic}), $\|\hat{f}_n-f\|_{L_2(\Omega)}=o_p(1)$ if $\lambda_n=o_p(1)$.
This explains why we can expect improved rates of convergence for the two terms on the left-hand side of (\ref{improvedbasic}). These rates can be obtained by employing additional algebraic calculations. We summarize our findings in Proposition \ref{th:improvedrate}. 

\begin{proposition}\label{th:improvedrate}
	Suppose there exists $v\in L_2(\Omega)$, such that
	\begin{eqnarray}
	f(x) =\int_\Omega \Phi(x-t)v (t)dt.
	\label{eq2.14}
	\end{eqnarray}
	Moreover, suppose the sequence of design points is quasi-uniform and the random error $e_i$'s are sub-Gaussian satisfying (\ref{eq2.2}).
	Then
	\begin{eqnarray}
	\begin{aligned}
	\| \hat f_n-f\|_{n}&=O_p\left(\lambda_n \vee n^{-\frac{1}{2}}\lambda_n^{-\frac{d}{4m}}\right),\\
	\| \hat f_n-f\|_{ \mathcal{N}_{\Phi}(\Omega)}&=O_p\left(\lambda^{1/2}_n\vee n^{-\frac{1}{2}}\lambda_n^{-\frac{2m+d}{4m}}\right)
	\end{aligned}.
	\label{eq2.15}
	\end{eqnarray}
\end{proposition}

\begin{proof}
	This result is a special case of Corollary \ref{coro:uniform} in Section \ref{sec:gaussian}.
\end{proof}

\begin{remark}
	The improved rates in Proposition \ref{th:improvedrate} are known; see \cite{blanchard2018optimal,gu2002smoothing,dicker2017kernel,guo2017learning,lin2017distributed} and the references therein. Despite these known rates, the conditions in Proposition \ref{th:improvedrate} differs from these works. These works focus on kernels represented by eigenvalues and eigenfunctions, and random designs. We consider Mat\'ern kernels and quasi-uniform designs, which are widely used in engineer and computer experiment applications. Also, the mathematical tools used here are different from those in the above works, and our analysis yields a stronger result, given in Theorem \ref{th:uniform}, which leads to an asymptotic theory for the K-O calibration estimator.
\end{remark}

In Proposition \ref{th:improvedrate}, since the design sequence is quasi-uniform, similar to Corollary \ref{Coro:rate}, we can apply Proposition \ref{Prop:utreras} to derive
\begin{eqnarray}
\| \hat f_n-f\|_{L_2(\Omega)}&=O_p\left(\lambda_n \vee n^{-\frac{1}{2}}\lambda_n^{-\frac{d}{4m}}\right).
\end{eqnarray}

\section{Calibration of Computer Models}\label{sec:gaussian}

In this section, we use the improved convergence theory established in Section \ref{sec ir} to study the asymptotic theory for the K-O method for the calibration of computer models.

In computer experiments, calibration is the activity of identifying the computer model parameters by matching the computer and physical outputs. Consider a physical experiment, with a vector of input variable denoted as $x$. To reduce the cost of the physical experiment, researchers often conduct a computer simulation to mimic the physical system as well. Usually, the computer code input consists of the physical input $x$ and model parameters $\theta$. The model parameters are not observed in the physical experiment; they commonly represent certain intrinsic attributes of the system. Here we consider only deterministic computer experiments, i.e., the computer output is a deterministic function of the inputs, denoted by $y^s(x,\theta)$.

In K-O's approach, the physical experimental data are modeled as
\begin{eqnarray}\label{KO1}
y_i=\xi(x_i)+e_i, i=1,\ldots,n 
\end{eqnarray}
where $\xi$ is an underlying function called the true process, $x_i$'s are fixed input points, and $e_i$'s are independent and identically distributed random error with mean zero.

Because the computer models are built under inevitable simplification and approximation, their outputs cannot coincide with the true process. \cite{kennedy2001bayesian} used the following model to link these functions
\begin{eqnarray}\label{theta0}
\xi(x)=y^s(x,\theta_0)+\delta(x), 
\end{eqnarray}
where $\theta_0$ is the ``optimal choice'' of the model parameter, and $\delta$ denotes the discrepancy function. The model (\ref{theta0}) is clearly non-identifiable, because both $\theta_0$ and $\delta$ are unknown. We refer to \cite{plumlee2015calibrating,tuo2019adjustments,tuo2015efficient,tuo2016theoretical,tuo2016prediction} for related theoretical discussions regarding the identifiability. Kennedy and O'Hagan \cite{kennedy2001bayesian} proposed to impose a Gaussian process prior on $\delta$ to facilitate the estimation of $\theta_0$.

Given the widespread use of the K-O method in computer experiments and related scientific and engineering problems, understanding the asymptotic properties of this method is of interest.
In this work, we \textit{do not} assume that $\delta$ (or $(\xi,y^s)$) is random, that is, we regard the Gaussian process modeling technique in the K-O's approach only as a computational method. This nonrandom model setting can be justified as follows. Because the computer code is deterministic, $y^s$ should be nonrandom. Also, the true process $\xi$ is usually presumed as nonrandom in industrial statistics, for example, in the response surface methodology \cite{wu2008experiments}. The main objective of this section is to study the asymptotic behavior of the K-O's calibration estimator under the above deterministic setting. Our findings in the section \textit{should not} be interpreted under the usual framework of Gaussian process regression, where the underlying function is truly random.


\subsection{A frequentist version of the Kennedy-O'Hagan's approach}

We consider estimating $\theta$ by maximizing the following ``likelihood function'':
\begin{multline}\label{likelihood}
L(\theta,\sigma^2,\tau^2)=\det(\sigma^2\mathbf{\Phi}+\tau^2 I)^{-1/2}\\
\times\exp\left\{-\frac{1}{2}(Y-y^s(X;\theta))^T(\sigma^2\mathbf{\Phi}+\tau^2 I)^{-1}(Y-y^s(X;\theta))\right\},
\end{multline}
where $\mathbf{\Phi}=(\Phi(x_i,x_j))_{i j}$, $Y=(y_1,\ldots,y_n)^T$, $y^s(X;\theta)=(y^s(x_1;\theta),\ldots,y^s(x_n;\theta))^T$ and $I$ denotes the identity matrix.

Under some extra conditions, (\ref{likelihood}) is indeed the likelihood function induced by the K-O approach. First, we suppose that $e_i$'s in (\ref{KO1}) follow the normal distribution $N(0,\tau^2)$ \footnote{In our theoretical analysis in Theorems \ref{th:uniform}-\ref{th:calibration}, we relax this assumption by incorporating sub-Gaussian noise.}, and we impose a Gaussian process prior on $y^s$. Second, suppose this Gaussian process has mean zero and covariance function $\sigma^2 \Phi(\cdot,\cdot)$. Here we assume that $\Phi$ is given. Then it is easily shown that the likelihood function of $(\theta,\sigma^2,\tau^2)$ is (\ref{likelihood}).

The MLEs of $\sigma^2$ and $\tau^2$ in (\ref{likelihood}) do not have explicit expressions. To ease the mathematical treatments, we suggest choosing the ratio $\lambda=\tau^2/(n\sigma^2)$ in a non-data-driven manner. We will show that, a deterministic choice of $\lambda$ (depending on $n$) can sufficiently lead to a desired asymptotic theory. Once $\lambda$ is given, we have the following simplified expression of $\hat{\theta}$:
%
\begin{eqnarray}\label{MLE}
\hat{\theta}=\operatorname*{argmin}_{\theta\in\Theta}(Y-y^s(X;\theta))^T(\mathbf{\Phi}+n\lambda I)^{-1}(Y-y^s(X;\theta)).
\end{eqnarray}

Our goal is to develop an asymptotic theory for $\hat{\theta}$, under the assumption that $\xi$ and $y^s$ are \textit{deterministic} functions. We call $\hat{\theta}$ the \textit{frequentist estimator of the K-O approach}. Of course, we adopt a totally different model setting compared with \cite{kennedy2001bayesian}. Computationally, the two methods are also different in the following aspects.
\begin{enumerate}
	\item In \cite{kennedy2001bayesian}, prior distributions are imposed on the parameters $\theta,\sigma^2,\tau^2$, and possibly the hyper-parameters associated with $\Phi$. In this work, we do not impose those distributions. Also, we do not introduce extra hyper-parameters on the kernel $\Phi$.
	\item In \cite{kennedy2001bayesian}, Bayesian analysis is conducted by calculating the posterior distribution. In this work, we focus on the MLE.
	\item In \cite{kennedy2001bayesian}, both $\sigma^2$ and $\tau^2$ are estimated from the data. In this work, we choose $\lambda=\tau^2/(n\sigma^2)$ in a non-data-driven manner to facilitate our mathematical analysis.
	\item In \cite{kennedy2001bayesian}, the computer model can be expensive to run, so that a surrogate model is introduced to reconstruct $y^s$. In this work, we assume that $y^s$ is a known function. This assumption is reasonable when the computer model is inexpensive.
\end{enumerate}


\subsection{Asymptotic theory}
The MLE estimator $\hat{\theta}$ in (\ref{MLE}) has a close relationship with the kernel ridge regression discussed in Section \ref{sec ir}. To see this, define
\begin{eqnarray*}
	\zeta^\theta(x)=\xi(x)-y^s(x;\theta), 
	\zeta^\theta_i=y_i-y^s(x;\theta).
\end{eqnarray*}
Let
\begin{equation*}
\hat{\zeta}^\theta=\operatorname*{argmin}_g \frac{1}{n}\sum_{i=1}^n (\zeta^\theta_i-g(x_i))^2+\lambda\|g\|^2_{\mathcal{N}_\Phi(\Omega)},
\end{equation*}
which is the kernel ridge regression estimator for $\zeta^\theta$.
The results are given in Theorem \ref{th:equivalence}.

\begin{theorem}\label{th:equivalence}
	The MLE estimator $\hat{\theta}$ can be represented by
	\begin{equation}\label{theta}
	\hat{\theta}=\operatorname*{argmin}_{\theta\in\Theta}\frac{1}{n}\sum_{i=1}^n (\zeta^\theta_i-\hat{\zeta}^\theta(x_i))^2+\lambda\|\hat{\zeta}^\theta\|^2_{\mathcal{N}_\Phi(\Omega)}.
	\end{equation}
\end{theorem}

To employ the theory developed in Section \ref{sec ir}, we assume that $\zeta^\theta$ lies in $\mathcal{N}_\Phi(\Omega)$, or a subspace of it. This assumption does not hold under a usual Gaussian process model, because the set $\mathcal{N}_\Phi(\Omega)$ has probability zero under the probability measure of the corresponding Gaussian process \cite{driscoll1973reproducing}. Our discussion, however, should not be affected because we are \textit{not} adopting a Gaussian process model. Also, we believe that $\zeta^\theta\in\mathcal{N}_\Phi(\Omega)$ is a reasonable assumption in the context of computer experiments, because the reproducing kernel Hilbert space is large enough, which covers all smooth functions.

For notational consistency with Section \ref{sec ir}, we write $\hat{\theta}$ as $\hat{\theta}_n$ to emphasis its dependency on $n$. Similarly, we write $\lambda$ as $\lambda_n$. Then (\ref{theta}) becomes
\begin{eqnarray*}
	\hat{\theta}_n=\operatorname*{argmin}_{\theta\in\Theta}\frac{1}{n}\sum_{i=1}^n (\zeta^\theta_i-\hat{\zeta}_n^\theta(x_i))^2+\lambda_n\|\hat{\zeta}_n^\theta\|^2_{\mathcal{N}_\Phi(\Omega)},
\end{eqnarray*}
with
\begin{equation*}
\hat{\zeta}^\theta_n=\operatorname*{argmin}_g \frac{1}{n}\sum_{i=1}^n (\zeta^\theta_i-g(x_i))^2+\lambda_n\|g\|^2_{\mathcal{N}_\Phi(\Omega)}.
\end{equation*}
Following the standard framework for establishing asymptotic theory for M-estimation, we should consider the limiting behavior of the objective function 
\begin{equation}\label{objective}
\frac{1}{n}\sum_{i=1}^n (\zeta^\theta_i-\hat{\zeta}_n^\theta(x_i))^2+\lambda_n\|\hat{\zeta}_n^\theta\|^2_{\mathcal{N}_\Phi(\Omega)}.
\end{equation}
Although this function is related to the kernel ridge regression, the standard rates of convergence for kernel ridge regression given by Corollary \ref{Coro:rate} are insufficient to provide an asymptotic result for $\hat{\theta}_n$. To see this, we note that according to Corollary \ref{Coro:rate}, the second term in (\ref{objective}) is merely known to be $O_p(\lambda_n)$. This error bound is too crude to ensure a convergence result for $\hat{\theta}_n$.

In contrast, if the conditions of Proposition \ref{th:improvedrate} are fulfilled, the improved rate of convergence gives the asymptotic representation
\begin{equation*}
\|\hat{\zeta}_n^\theta\|^2_{\mathcal{N}_\Phi(\Omega)}=\|\zeta^\theta\|^2_{\mathcal{N}_\Phi(\Omega)}+O_p(\lambda_n),
\end{equation*}
which gives a much finer error bound.
Thanks to the improved rates of convergence, we can establish an asymptotic theory for $\hat{\theta}_n$.

We first consider the prediction problem: how accurate $\hat{\zeta}^\theta_n$ can approximate $\zeta^\theta$ in a uniform sense. The result, which is a generalization of Proposition \ref{th:improvedrate}, is given by Theorem \ref{th:uniform}. As in Section \ref{sec ir}, we assume that the reproducing kernel Hilbert space $\mathcal{N}_\Phi(\Omega)$ is equal to some (fractional) Sobolev space $H^m(\Omega)$ with equivalent norms, for some $m>d/2$. Specifically, if $\Phi$ is a Mat\'ern kernel in (\ref{matern}), then $m=\nu+d/2$.

In Theorem \ref{th:uniform}, we pursue non-asymptotic error bounds, that is, the sample size $n$ is assumed to be fixed rather than tending to infinity.
In the rest of this article, we use $c_1,c_2,c_3,\ldots$ to denote universal positive constants. They are independent of $n$. They may depend on $m,d,\Omega$, and the quasi-uniformity constant $B$ in (\ref{quasiuniform}), but are independent of the specific collocation scheme of the design points. For simplicity, we may use the same $c_i$ in different places to denote different constants.

\begin{theorem}\label{th:uniform}
	As in Proposition \ref{th:improvedrate}, we suppose the set of design points is quasi-uniform, i.e., (\ref{quasiuniform}) holds.
	Suppose for each $\theta\in\Theta$, $\zeta^\theta\in \mathcal{N}_{\Phi}(\Omega)$, and there exists $v_{\theta} \in L_2(\Omega)$, such that%
	\begin{eqnarray}\label{condition1}
	\begin{split}
	\zeta^\theta(x) =\int_\Omega \Phi(x-t)v_{\theta} (t)dt,\\
	\bar{v}:=\sup_{\theta\in\Theta} \| v_{\theta} \|_{L_2(\Omega)}<+\infty.
	\end{split}
	\end{eqnarray}
	Then for $n>c_1$, the following two inequalities
	\begin{eqnarray*}
		\sup_{\theta\in\Theta}\|\hat{\zeta}_n^\theta-{\zeta}^\theta\|_n&\leq& c_2 \bar{v} \lambda_n \vee c_3 t n^{-\frac{1}{2}}\lambda_n^{-\frac{d}{4m}},\\
		\sup_{\theta\in\Theta}\|\hat{\zeta}_n^\theta-{\zeta}^\theta\|_{\mathcal{N}_\Phi(\Omega)}&\leq& c_4 \bar{v} \lambda_n^{\frac{1}{2}} \vee c_5 t n^{-\frac{1}{2}}\lambda_n^{-\frac{2m+d}{4m}},
	\end{eqnarray*}
	hold simultaneously on the event
	\begin{equation}\label{eventA}
	A_t:=	\left\{\sup_{g \in \mathcal{N}_{\Phi}(\Omega)}\frac{|\langle e,g\rangle_n|}{\| g \|_n^{1-\frac{d}{2m}} \| g \|_{\mathcal{N}_{\Phi}(\Omega)}^{\frac{d}{2m}} }\leq t n^{-1/2}\right\}.
	\end{equation}
\end{theorem}

Condition (\ref{condition1}) is a uniform version of the condition (\ref{T}), because in (\ref{condition1}) we require not only the existence of $v_\theta\in L_2(\Omega)$, but also the uniform boundedness of their $L_2$ norms. Suppose a Mat\'ern kernel in (\ref{matern}) with $m=\nu+d/2$ is used. Theorem \ref{th:existence} shows that (\ref{T}) is equivalent to $\|f_e\|_{H^{2m}(\mathbb{R}^d)}<\infty$. From the proof of Theorem \ref{th:existence}, one can justify that (\ref{condition1}) is equivalent to $\sup_{\theta\in\Theta}\|\zeta^\theta_e\|_{H^{2m}(\mathbb{R}^d)}<\infty$.
From Theorem \ref{th:uniform}, we can establish
the asymptotic rates of convergence as given in Corollary \ref{coro:uniform}.

\begin{corollary}\label{coro:uniform}
	Suppose $e_i$'s are sub-Gaussian. Then under the conditions of Theorem \ref{th:uniform}, we have the rates of convergence
	\begin{eqnarray*}
		\sup_{\theta\in\Theta}\|\hat{\zeta}_n^\theta-{\zeta}^\theta\|_n&=& O_p\left(\lambda_n \vee  n^{-\frac{1}{2}}\lambda_n^{-\frac{d}{4m}}\right),\\
		\sup_{\theta\in\Theta}\|\hat{\zeta}_n^\theta-{\zeta}^\theta\|_{\mathcal{N}_\Phi(\Omega)}&=& O_p\left( \lambda_n^{\frac{1}{2}}\vee  n^{-\frac{1}{2}}\lambda_n^{-\frac{2m+d}{4m}}\right).
	\end{eqnarray*}
\end{corollary}

\begin{proof}
	According to Lemma \ref{lemma:empirical},
	$A_t$ has probability at least $1-c_1\exp\{-c_2t^2\}$ for all $t>c_3$, which tends to one as $t\rightarrow +\infty$. The rates then follow from Theorem \ref{th:uniform}.
\end{proof}

Next we state the convergence results for $\hat{\theta}_n$. We will show that under certain conditions, $\hat{\theta}_n$ will tend to
\begin{eqnarray}\label{thetaprime}
\theta'=\operatorname*{argmin}_{\theta\in\Theta}\|\zeta^\theta\|_{\mathcal{N}_\Phi(\Omega)},
\end{eqnarray}
as $n\rightarrow\infty$. Here we only present the error bound of $\|\hat{\theta}_n-\theta'\|$ for the case $\lambda_n^{-1}=O_p(n^{\frac{2m}{4m+d}})$, because this case gives the best rate of convergence. By using similar but more cumbersome mathematical analysis, we can show that $\hat{\theta}_n$ converges to $\theta'$ if $\lambda_n^{-1}=o_p(n^{\frac{2m}{2m+d}})$. The general error bounds are more complicated and we choose not to pursue them here.

\begin{theorem}\label{th:calibration}
	Suppose the conditions of Theorem \ref{th:uniform} are fulfilled. In addition, we suppose that $\theta'$ is the unique solution to (\ref{thetaprime}). Moreover, there exists constants $a_2,a_3,\gamma>0$ such that
	\begin{eqnarray}\label{curvature}
	\|\zeta^\theta\|^2_{\mathcal{N}_\Phi(\Omega)}-\|\zeta^{\theta'}\|^2_{\mathcal{N}_\Phi(\Omega)}\geq a_2\min\{\|\theta-\theta'\|^\gamma,a_3\},
	\end{eqnarray}
	for all $\theta\in\Theta$, where $\|\cdot\|$ denotes the Euclidean distance. Let $A_t$ be the event defined in (\ref{eventA}), and 
	\begin{eqnarray}\label{lambdabound}
	\lambda_n^{\frac{4m+d}{4m}}>a_1 \bar{v}^{-1}tn^{-1/2},
	\end{eqnarray}
	for some $a_1>0$.
	If $\bar{v}^2\lambda_n<c_1$, then on the event $A_t$,
	\begin{eqnarray*}
		\|\hat{\theta}_n-\theta'\|\leq c_3 \bar{v}^{2/\gamma} \lambda_n^{1/\gamma}.
	\end{eqnarray*}
\end{theorem}

\begin{remark}
	Suppose $h(\theta):=\|\zeta^\theta\|^2_{\mathcal{N}_\Phi(\Omega)}-\|\zeta^{\theta'}\|^2_{\mathcal{N}_\Phi(\Omega)}$ is continuously twice differentiable around $\theta'$. Then we can apply Taylor's theorem to conclude that (\ref{curvature}) holds with $\gamma=2$.
\end{remark}

\begin{corollary}
	Under the conditions of Theorem \ref{th:calibration} and $\lambda_n^{-1}=O(n^{\frac{2m}{4m+d}})$, we have the rate of convergence $\|\hat{\theta}_n-\theta'\|=O_p(\lambda_n^{1/\gamma})$. Specifically, if $h(\theta):=\|\zeta^\theta\|^2_{\mathcal{N}_\Phi(\Omega)}-\|\zeta^{\theta'}\|^2_{\mathcal{N}_\Phi(\Omega)}$ is continuously twice differentiable around $\theta'$, then $\|\hat{\theta}_n-\theta'\|=O_p(\lambda_n^{1/2})$.
\end{corollary}

\begin{proof}
	According to Lemma \ref{lemma:empirical},
	$A_t$ has probability at least $1-c_1\exp\{-c_2t^2\}$ for all $t>c_3$, which tends to one as $t\rightarrow +\infty$. The rate then follows from Theorem \ref{th:calibration}.
\end{proof}


\begin{remark}
	\cite{tuo2016theoretical} observed that under certain conditions, the limit value of the K-O method is  $\theta'$ defined in (\ref{thetaprime}), i.e., $\theta_0=\theta'$.
	In Theorem 4.2 of \cite{tuo2016theoretical}, they prove the limit result when the physical observations $y_i$ have no random error, i.e., $e_i$'s in (\ref{KO1}) are zero. In Tuo-Wu's result, the condition (\ref{condition1}) is also necessary in the mathematical treatments. In Theorem \ref{th:calibration} of this paper, we generalize the Tuo-Wu theory by assuming that $e_i$'s are independent and identically distributed sub-Gaussian random variables, and obtain the rate of convergence. Given the fact that physical responses are always subject to random noise, Theorem \ref{th:calibration} in this paper is much more useful than Theorem 4.2 of \cite{tuo2016theoretical} for practical applications. Therefore, the result we obtain here can be viewed as a substantial improvement over the Tuo-Wu theory. 
\end{remark}

\section{A simulation study}\label{sec:simulation}

The main objective of this section is to verify the rate of convergence given by Theorem \ref{th:calibration} in a simulation study. Theorem \ref{th:calibration} asserts that under certain conditions and $\lambda_n\sim n^{-\frac{2m}{4m+d}}$, we have the rate of convergence $\|\hat{\theta}_n-\theta'\|=O_p(\lambda_n^{1/\gamma})$. Specifically, if $h(\theta):=\|\zeta^\theta\|^2_{\mathcal{N}_\Phi(\Omega)}-\|\zeta^{\theta'}\|^2_{\mathcal{N}_\Phi(\Omega)}$ is continuously twice differentiable around $\theta'$, then 
\begin{eqnarray}\label{rate}
\|\hat{\theta}_n-\theta'\|=O_p(\lambda_n^{1/2})=O_p(n^{-\frac{m}{4m+d}}).
\end{eqnarray}

Our goal is to conduct a numerical study to verify whether the rate of convergence $O_p(n^{-\frac{m}{4m+d}})$ is sharp.
To this end, we take the logarithm on both sides of (\ref{rate}) to get
\begin{eqnarray*}
	\log \|\hat{\theta}_n-\theta'\|\lesssim -\frac{m}{4m+d}\log n+ c.
\end{eqnarray*}
This inspires us to consider a set of sample sizes, denoted as $\{n_1,\ldots,n_m\}$. Then for each $n_j$, we conduct an independent simulation and computer $E_j=\|\hat{\theta}_{n_j}-\theta'\|$. Next we consider the regression problem given by
\begin{eqnarray}\label{appro}
\log E_j=a+b\log n_j+e_j, ~~ j=1,\ldots,m. 
\end{eqnarray}
We estimate the regression coefficients $(a,b)$ by the least squares method and denote the estimator as $(\hat{a},\hat{b})$.
Then we can regard $O_p(n^{\hat{b}})$ as the estimated rate of convergence. We shall check whether $\hat{b}$ is close to $-\frac{m}{4m+d}$, the theoretical rate of convergence asserted by Theorem \ref{th:calibration}.

%

In our simulation study, we need to find functions that satisfy the condition (\ref{eq2.14})  Suppose $\Phi(x)$ is the exponential kernel function $\Phi(x)=\exp\{-|x|\}$, which is also the Mat\'ern kernel function (\ref{matern}) with $\phi=1$ and $\nu=0.5$, and the experimental region $\Omega=[-1,1]$. The corresponding Sobolev space is $H^1[-1,1]$.


Suppose the true process $\xi$ is
\begin{eqnarray*}
\xi(x)=\int_{-1}^1\Phi(x-t)\Phi(t)dt=e^{-|x|}+|x|e^{-|x|}-e^{x-2}/2-e^{-(x+2)}/2,
\end{eqnarray*}
and the computer model is
$$y^s(x,\theta)=\zeta(x)-\int_{-1}^1 z(x-y)(\theta y^2+0.8)dy,$$
where $\theta$ is the model parameter to be calibrated.

Clearly, the discrepancy function $\zeta^\theta(x)=\int_{-1}^1 z(x-y)(\theta y^2+0.8)dy$ satisfies all conditions of Theorem \ref{th:calibration}. The identity (\ref{normequality}) implies
$$\|\zeta^\theta\|^2_{\mathcal{N}_\Phi(\Omega)}=\int_{-1}^1\int_{-1}^1 (\theta x^2+0.8)\Phi(x)\Phi(x-y)\Phi(y)(\theta y^2+0.8)d x d y.  $$
By numerical search, we find that, as a function of $\theta$, $\|\zeta^\theta\|^2_{\mathcal{N}_\Phi(\Omega)}$ is minimized at $\theta'=0.672$.

Suppose we observe data 
$$y_i=z(x_i)+e_i,i=1,\ldots,n,$$
where $e_i$'s are independent and identically distributed random errors following $ N(0,\tau^2)$.

Now we can compute $\hat{\theta}$ in (\ref{theta}). Following the theoretical guidance in Theorem \ref{th:calibration}, we choose $\lambda=n^{-2m/(4m+d)}$.
To estimate the regression coefficient in (\ref{appro}), we choose $n$ different Sobol designs \cite{santner2013design} with sample sizes $n_j=20j,j=1,\ldots,30$. For each $j$, we repeat the simulation 100 times and calculate the Monte Carlo sample mean to reduce the random error.

\begin{figure}[htbp]
  \centering
  \label{fig:a}\includegraphics{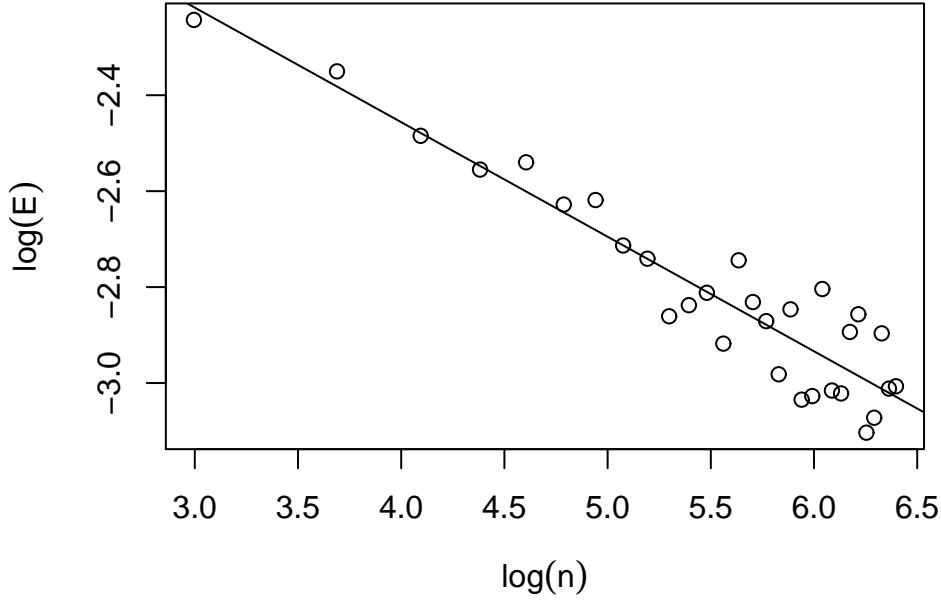}
  \caption{The scattered plot and the regression line of the simulated data.}
  \label{fig:testfig}
\end{figure}

The scattered plot of $\log E_j$ against $\log n_j$ is shown in Figure \ref{fig:a}. The estimated regression coefficient is -0.20058, which closely agrees with our theoretical assertion -0.2.

\section{Discussion}\label{sec:discussion}

In this work, we obtain some new results on the improved rates of convergence for kernel ridge regression. We apply this theory to study the asymptotic properties of the K-O calibration method for computer experiments. This new result generalizes the work by \cite{tuo2016theoretical}.

Several related problems that can be studied in the future.
In this article, we suppose the design set $\{x_1,\ldots,x_n\}$ is fixed and quasi-uniform. A further question is whether the improved rates still hold if the design points are random samples; for instance, if the design points are independent and follow the uniform distribution over $\Omega$. 

As is discussed in Section \ref{sec ir}, compared to the existing results,  the bias of the kernel ridge regression estimator is reduced by imposing the condition (\ref{eq2.14}), while the variance remains the same. Improved rates of convergence are achieved by rebalancing the bias and the variance. In other words, the choice of the smoothing parameters $\lambda_n$ is crucial in achieving the optimal rate of convergence. 
Suppose condition (\ref{eq2.14}) is fulfilled. Proposition \ref{th:improvedrate} implies that the optimal tuning parameter is $\lambda_n\sim n^{-\frac{2m}{4m+d}}$. If condition (\ref{eq2.14}) is not satisfied, we should return to the classic results given by Proposition \ref{th:existing}. In this case, the optimal tuning parameter is $\lambda_n\sim n^{-\frac{2m}{2m+d}}$, and $\lambda_n\sim n^{-\frac{2m}{4m+d}}$ would render a suboptimal rate of convergence.
In most practical scenarios, we do not know whether the condition (\ref{eq2.14}) holds or not. Therefore, there is no \textit{a priori} optimal choice of $\lambda_n$. One would ask whether the optimal order of magnitude for $\lambda_n$ can be obtained by a data-driven approach. We conjecture that model selection criteria like the generalized cross validation \cite{wahba1990spline} can automatically adapt an optimal choice of $\lambda_n$.

\appendix

\section{Technical Proofs}\label{App:proof}

\begin{proof}[Proof of Theorem \ref{th:existence}]
	Without loss of generality, we can assume that the scale parameter $\phi$ in (\ref{matern}) is $1/(2\sqrt{\nu})$, because otherwise we can stretch the region $\Omega$ to make this happen. In this situation, the Mat\'ern kernel becomes
	\begin{equation*}\label{Phinu}
	\frac{1}{\Gamma(\nu)2^{\nu-1}}|x|^\nu K_\nu(|x|).
	\end{equation*}
	
	Suppose $f(x)=\int_{\Omega} \Phi(x-t) v(t)$ with $v\in L_2(\Omega)$. It can be justified that $f_e=\int_{\Omega}\Phi(x-t)v(t)d t$ for $x\in\mathbb{R}^d$. See \cite{schaback1999improved} for details. Define 
	\begin{eqnarray*}
		v_e(x)=\begin{cases}
			v(x), & x\in\Omega,\\
			0, & x\notin \Omega.
		\end{cases}
	\end{eqnarray*}
	Clearly $f_e(x)=\int_{\mathbb{R}^d}\Phi(x-t)v_e(t)d t$. For $h\in L_2(\mathbb{R}^d)$, denotes its Fourier transform and inverse Fourier transform by $\mathcal{F}(h)$ and $\mathcal{F}^{-1}(h)$, respectively.
	Then by the convolution theorem, $\mathcal{F}(f_e)=(2\pi)^{d/2}\mathcal{F}(\Phi)\mathcal{F}(v_e)$. Direct calculations \cite{stein2012interpolation,wendland2004scattered} give
	\begin{eqnarray}\label{Fourier}
	\mathcal{F}(\Phi)(\omega)&=&2^{d/2}\frac{\Gamma(\nu+d/2)}{\Gamma(\nu)}(1+\|\omega\|^2)^{-(\nu/2+d/4)}\nonumber\\
	&:=&C_0(1+\|\omega\|^2)^{-m/2}.
	\end{eqnarray}
	Note that $\mathcal{F}(f_e)/\mathcal{F}(\Phi)=\mathcal{F}(v_e)\in L_2(\mathbb{R}^d)$, which gives
	\begin{eqnarray}\label{Sobolev}
	\int_{\mathbb{R}^d}(1+\|\omega\|^2)^{m} |\mathcal{F}(f_e)(\omega)|^2 d\omega < +\infty.
	\end{eqnarray}
	According to Paragraph 7.62 of \cite{adams2003sobolev}, (\ref{Sobolev}) is equivalent to $f_e\in H^{2m}(\mathbb{R})$.
	
	Suppose $f_e\in H^{2m}(\mathbb{R})$. Then by (\ref{Sobolev}) and (\ref{Fourier}), we have $\mathcal{F}(f_e)/\mathcal{F}(\Phi)\in L_2(\mathbb{R}^d)$. Theorem 4.3 of \cite{schaback1999improved} proves that in this case, $h_f:=\mathcal{F}^{-1}(\mathcal{F}(f_e)/\mathcal{F}(\Phi))$ vanishes almost everywhere outside $\Omega$. Then according to the convolution theorem, $v:=h_f|_\Omega$ satisfies (\ref{integralequation}).
\end{proof} 

\begin{lemma}\label{lemma:empirical}	
	Suppose $\{x_1,\ldots,x_n\}\subset \Omega$,
	$e_1,\ldots,e_n$ are independent and identically distributed random variables which are sub-Gaussian. Then for all $t>c_1$, we have
	\begin{eqnarray}
	\sup_{g \in \mathcal{N}_{\Phi}(\Omega)}\frac{|\langle e,g\rangle_n|}{\| g \|_n^{1-\frac{d}{2m}} \| g \|_{\mathcal{N}_{\Phi}(\Omega)}^{\frac{d}{2m}} }\leq t n^{-1/2},
	\label{eq2.5}
	\end{eqnarray}
	with probability at least $1-c_2\exp\{-c_3 t^2\}$, where $\langle e,g\rangle_n$ is defined in (\ref{empiricalinnerproduct}).
\end{lemma}

\begin{proof}
	For $g\in \mathcal{N}_\Phi(\Omega)$, let $h=g/\|g\|_{\mathcal{N}_\Phi(\Omega)}$. It is easily verified that
	\begin{eqnarray*}
		\frac{|\langle e,g\rangle_n|}{\| g \|_n^{1-\frac{d}{2m}} \| g \|_{\mathcal{N}_{\Phi}(\Omega)}^{\frac{d}{2m}} }=\frac{|\langle e,h\rangle_n|}{\| h \|_n^{1-\frac{d}{2m}} }.	
	\end{eqnarray*}
	Let $\mathcal{H}=\{h\in\mathcal{N}_\Phi(\Omega):\|h\|_{\mathcal{N}_\Phi(\Omega)}=1\}$.
	Noting that $\mathcal{N}_{\Phi}(\Omega)$ can be embedded into $H^m(\Omega)$, we can use the metric entropy of the Sobolev spaces \cite{Edmunds1996Function,tuo2015efficient} to find an upper bound of the metric entropy of $\mathcal{H}$ as
	\begin{eqnarray}\nonumber
	H(\epsilon,\mathcal{H},\| \cdot\|_n)\leq c_4 \epsilon ^{-d/m}.
	\end{eqnarray}
	We refer to \cite{van1996weak} for the definition and detailed discussions about the metric entropy of a function space. The remainder of the proof follows by invoking the concentration inequality given by Corollary 14.6 of \cite{buhlmann2011statistics}.
\end{proof}

\begin{proof}[Proof of Theorem \ref{th:equivalence}]
	Using (\ref{representer}) and (\ref{ridge}), we find that $\hat{\zeta}^\theta$ can be expressed by
	\begin{eqnarray*}
		\hat{\zeta}^\theta=\sum_{i=1}^n c_i^\theta \Phi(x-x_i),
	\end{eqnarray*}
	with $c^\theta=(c_1^\theta,\ldots,c_n^\theta)^T$ defined as
	\begin{eqnarray*}
		c=(\mathbf{\Phi}+n\lambda I_n)^{-1} Y_\theta,\end{eqnarray*}
	where $\mathbf{\Phi}=(\Phi(x_i,x_j))_{i j}$ and $Y_\theta=(\zeta^\theta_1,\ldots,\zeta^\theta_n)^T$. The norm of $\hat{\zeta}^\theta$ in $\mathcal{N}_\Phi(\Omega)$ can be calculated using (\ref{eq1.2}), given by
	\begin{eqnarray}\label{native}
	\|\hat{\zeta}^\theta\|^2_{\mathcal{N}_\Phi(\Omega)}=Y_\theta^T (\mathbf{\Phi}+n\lambda I_n)^{-1} \mathbf{\Phi} (\mathbf{\Phi}+n\lambda I_n)^{-1} Y_\theta.
	\end{eqnarray}
	Also, $\hat{\zeta}^\theta(X):=(\hat{\zeta}^\theta(x_1),\ldots,\hat{\zeta}^\theta(x_n))^T$ can be expressed by
	\begin{eqnarray*}
		\hat{\zeta}^\theta(X)=\mathbf{\Phi}(\mathbf{\Phi}+n\lambda I_n)^{-1} Y_\theta,
	\end{eqnarray*}
	which yields
	\begin{eqnarray*}
		Y_\theta-\hat{\zeta}^\theta(X)&=&(I-\mathbf{\Phi}(\mathbf{\Phi}+n\lambda I_n)^{-1})Y_\theta\nonumber\\&=&n\lambda (\mathbf{\Phi}+n\lambda I_n)^{-1} Y_\theta.
	\end{eqnarray*}
	Thus
	\begin{eqnarray}\label{n}
	\frac{1}{n}\sum_{i=1}^n (\zeta_i^\theta-\hat{\zeta}^\theta(x_i))^2&=&\frac{1}{n}(Y_\theta-\hat{\zeta}^\theta(X))^T(Y_\theta-\hat{\zeta}^\theta(X))\nonumber\\
	&=&n\lambda^2Y_\theta^T (\mathbf{\Phi}+n\lambda I_n)^{-2} Y_\theta.
	\end{eqnarray}
	From (\ref{native}) and (\ref{n}) we obtain
	\begin{eqnarray*}
		\frac{1}{n}\sum_{i=1}^n (\zeta_i^\theta-\hat{\zeta}^\theta(x_i))^2+\lambda\|\hat{\zeta}^\theta\|^2_{\mathcal{N}_\Phi(\Omega)}=\lambda Y_\theta^T(\mathbf{\Phi}+n\lambda I_n)^{-1} Y_\theta,
	\end{eqnarray*}
	which implies the desired results.
\end{proof}

\begin{lemma}\label{lemma:order}
	Let $\lambda, P, Q, s, t$ be nonnegative numbers and $n$ be a positive integer. If there exist constants $a_1,a_2,a_3,a_4>0$, such that
	\begin{eqnarray}\label{PQ}
	P^2+ a_1\lambda Q^2\leq a_2 s \lambda P+a_3 t n^{-\frac{1}{2}}P^{1-\frac{d}{2m}} Q^{\frac{d}{2m}},
	\end{eqnarray}
	then we have
	\begin{eqnarray*}
		\begin{aligned}
			P&\leq  c_1 s\lambda\vee c_2 t n^{-\frac{1}{2}}\lambda^{-\frac{d}{4m}} ,\\
			Q&\leq c_3 s\lambda^{\frac{1}{2}}\vee c_4 t n^{-\frac{1}{2}}\lambda^{-\frac{2m+d}{4m}} .
		\end{aligned}
	\end{eqnarray*}
	Here $c_1,c_2,c_3,c_4$ are independent of $P, Q,\lambda, s$, and $t$.
\end{lemma}

\begin{proof}
	Clearly, (\ref{PQ}) implies either
	\begin{equation*}
	P^2+ a_1\lambda Q^2\leq 2 a_2 s \lambda P,
	\end{equation*}
	or
	\begin{equation*}
	P^2+ a_1\lambda Q^2\leq 2 a_3 n^{-1/2}P^{1-\frac{d}{2m}} Q^{\frac{d}{2m}}.
	\end{equation*}
	Next we consider these two cases separately.
	
	Case I. Suppose $P^2+ a_1\lambda Q^2\leq 2 a_2 s \lambda P$. Then we have
	\begin{eqnarray*}
		\begin{aligned}
			P^2&\leq 2 a_2 s \lambda P,\\
			a_1\lambda Q^2&\leq 2 a_2 s \lambda P,
		\end{aligned}
	\end{eqnarray*}
	which yields
	\begin{eqnarray}\label{le2part1}
	\begin{aligned}
	P&\leq 2 a_2 s \lambda ,\\
	Q&\leq 2 a_2 a_1^{-1/2} s \lambda^{1/2}.
	\end{aligned}
	\end{eqnarray}
	
	Case II. Suppose $P^2+ a_1\lambda Q^2\leq 2 a_3 n^{-1/2}P^{1-\frac{d}{2m}} Q^{\frac{d}{2m}}$. Then we have
	\begin{eqnarray}\label{PQineq2}
	\begin{aligned}
	P^2&\leq 2 a_3 t n^{-1/2}P^{1-\frac{d}{2m}} Q^{\frac{d}{2m}},\\
	a_1\lambda Q^2&\leq 2 a_3 t n^{-1/2}P^{1-\frac{d}{2m}} Q^{\frac{d}{2m}}.
	\end{aligned}
	\end{eqnarray}
	By elementary calculations, we find that (\ref{PQineq2}) implies
	\begin{eqnarray}\label{le2part2}
	\begin{split}
	&P\leq c_5 t n^{-\frac{1}{2}}\lambda^{-\frac{d}{4m}},\\
	&Q\leq c_6 t n^{-\frac{1}{2}}\lambda^{-\frac{2m+d}{4m}}.
	\end{split}
	\end{eqnarray}
	
	The desired results then follows by combining (\ref{le2part1}) and (\ref{le2part2}).
\end{proof}

\begin{proof}[Proof of Theorem \ref{th:uniform}]
	Following similar arguments as those in (\ref{eq2.20})-(\ref{improvedbasic}), we can deduce the improved basic inequality
	\begin{eqnarray}\label{uniformbasicineq}
	\begin{split}
	&\| \zeta^\theta-\hat{\zeta}^\theta_n\|_n^2+\lambda_n \| \zeta^\theta-\hat{\zeta}^\theta_n\|^2_{\mathcal{N}_\Phi(\Omega)}\\
	&\leq 2\langle e,\hat{\zeta}^\theta_n-\zeta^\theta\rangle_n+2\lambda_n\|v_\theta\|_{L_2(\Omega)}\| \zeta^\theta-\hat{\zeta}^\theta_n\|_{L_2(\Omega)},
	\end{split}	
	\end{eqnarray}
	which holds for all $\theta\in\Theta$.
	
	It follows from Proposition \ref{Prop:utreras} and (\ref{hbound}) that, for sufficiently large $n$,
	\begin{eqnarray}
	\|\zeta^\theta-\hat{\zeta}_n^\theta\|_{L_2(\Omega)}&\leq& c_6\sqrt{\|\zeta^\theta-\hat{\zeta}_n^\theta\|^2_n+n^{-\frac{2m}{d}}\|\zeta^\theta-\hat{\zeta}_n^\theta\|^2_{H^m(\Omega)}}\nonumber\\
	&\leq& c_6\left\{\|\zeta^\theta-\hat{\zeta}^\theta_n\|_n+n^{-\frac{m}{d}}\|\zeta^\theta-\hat{\zeta}^\theta_n\|_{H^m(\Omega)}\right\}\nonumber\\
	&\leq& c_6\|\zeta^\theta-\hat{\zeta}^\theta_n\|_n+c_7 n^{-\frac{m}{d}}\|\zeta^\theta-\hat{\zeta}^\theta_n\|_{\mathcal{N}_\Phi(\Omega)},\label{l2ton}
	\end{eqnarray}
	where the last inequality follows from the assumption that $\|\cdot\|_{H^m(\Omega)}$ and $\|\cdot\|_{\mathcal{N}_\Phi(\Omega)}$ are equivalent. 
	
	Combining (\ref{uniformbasicineq}), (\ref{l2ton}) and the condition $\bar{v}=\sup_{\theta\in\Theta}\|v_\theta\|_{L_2(\Omega)}<+\infty$ yields
	\begin{eqnarray}\label{threeterms}
	\begin{split}
	&\| \zeta^\theta-\hat{\zeta}^\theta_n\|_n^2+\lambda_n \| \zeta^\theta-\hat{\zeta}^\theta_n\|^2_{\mathcal{N}_\Phi(\Omega)}\leq 2\langle e,\hat{\zeta}^\theta_n-\zeta^\theta\rangle_n\\
	&+2c_6\lambda_n\bar{v}\| \zeta^\theta-\hat{\zeta}^\theta_n\|_{L_2(\Omega)}+2c_7\lambda_n\bar{v}n^{-\frac{m}{d}}\|\zeta^\theta-\hat{\zeta}_n^\theta\|_{\mathcal{N}_\Phi(\Omega)}.
	\end{split}	
	\end{eqnarray}
	
	Now we consider three different cases.
	
	Case I. Suppose $n^{-\frac{m}{d}}\|\zeta^\theta-\hat{\zeta}^\theta_n\|_{\mathcal{N}_\Phi(\Omega)}\leq \|\zeta^\theta-\hat{\zeta}^\theta_n\|_{L_2(\Omega)}$. Then we obtain from (\ref{threeterms}) that
	\begin{eqnarray}\label{caseI}
	\begin{split}
	&\| \zeta^\theta-\hat{\zeta}^\theta_n\|_n^2+\lambda_n \| \zeta^\theta-\hat{\zeta}^\theta_n\|^2_{\mathcal{N}_\Phi(\Omega)}\leq 2\langle e,\hat{\zeta}^\theta_n-\zeta^\theta\rangle_n\\
	&+2(c_6+c_7)\lambda_n\bar{v}\| \zeta^\theta-\hat{\zeta}^\theta_n\|_{L_2(\Omega)}.
	\end{split}
	\end{eqnarray}
	
	Case II. Suppose $n^{-\frac{m}{d}}\|\zeta^\theta-\hat{\zeta}^\theta_n\|_{\mathcal{N}_\Phi(\Omega)}> \|\zeta^\theta-\hat{\zeta}^\theta_n\|_{L_2(\Omega)}$ and $4c_7\bar{v} n^{-\frac{m}{d}}\leq \|\zeta^\theta-\hat{\zeta}^\theta_n\|_{\mathcal{N}_\Phi(\Omega)}$. Then we can cancel the term $\lambda_n \| \zeta^\theta-\hat{\zeta}^\theta_n\|^2_{\mathcal{N}_\Phi(\Omega)}/2$ from both sides of (\ref{threeterms}) and get
	\begin{eqnarray}\label{caseII}
	\begin{split}
	&\| \zeta^\theta-\hat{\zeta}^\theta_n\|_n^2+\frac{1}{2}\lambda_n \| \zeta^\theta-\hat{\zeta}^\theta_n\|^2_{\mathcal{N}_\Phi(\Omega)}\leq 2\langle e,\hat{\zeta}^\theta_n-\zeta^\theta\rangle_n\\
	&+2c_6\lambda_n\bar{v}\| \zeta^\theta-\hat{\zeta}^\theta_n\|_{L_2(\Omega)}.
	\end{split}
	\end{eqnarray}
	
	Case III. Suppose $n^{-\frac{m}{d}}\|\zeta^\theta-\hat{\zeta}^\theta_n\|_{\mathcal{N}_\Phi(\Omega)}> \|\zeta^\theta-\hat{\zeta}^\theta_n\|_{L_2(\Omega)}$ and $4c_7\bar{v} n^{-\frac{m}{d}}> \|\zeta^\theta-\hat{\zeta}^\theta_n\|_{\mathcal{N}_\Phi(\Omega)}$. It follows directly from this assumption that
	\begin{eqnarray*}
		\|\zeta^\theta-\hat{\zeta}^\theta_n\|_{\mathcal{N}_\Phi(\Omega)}<4c_7\bar{v} n^{-\frac{m}{d}},\\
		\|\zeta^\theta-\hat{\zeta}^\theta_n\|_{L_2(\Omega)}<4c_7\bar{v} n^{-\frac{2m}{d}},
	\end{eqnarray*}
	from which we have already arrived at the desired results.
	
	Now we only need to consider the first two cases. Clearly, both (\ref{caseI}) and (\ref{caseII}) can be expressed as
	\begin{eqnarray}\label{l2basicineq}
	\begin{split}
	&\| \zeta^\theta-\hat{\zeta}^\theta_n\|_n^2+c_8\lambda_n \| \zeta^\theta-\hat{\zeta}^\theta_n\|^2_{\mathcal{N}_\Phi(\Omega)}\leq \\
	&2\langle e,\hat{\zeta}^\theta_n-\zeta^\theta\rangle_n+c_9\lambda_n\bar{v}\| \zeta^\theta-\hat{\zeta}^\theta_n\|_{L_2(\Omega)}.
	\end{split}
	\end{eqnarray}
	%
	On the event $A_t$, we have the inequality
	\begin{eqnarray}
	&&|\langle e,\hat{\zeta}^\theta_n-\zeta^\theta\rangle_n|\nonumber\\
	&\leq& \sup_{g \in \mathcal{N}_{\Phi}(\Omega)}\frac{|\langle e,g\rangle_n|}{\| g \|_n^{1-\frac{d}{2m}} \| g \|_{\mathcal{N}_{\Phi}(\Omega)}^{\frac{d}{2m}} } \cdot \| \hat{\zeta}^\theta_n-\zeta^\theta \|_n^{1-\frac{d}{2m}} \| \hat{\zeta}^\theta_n-\zeta^\theta \|_{\mathcal{N}_{\Phi}(\Omega)}^{\frac{d}{2m}}\nonumber\\
	&\leq& t n^{-1/2}\| \hat{\zeta}^\theta_n-\zeta^\theta \|_n^{1-\frac{d}{2m}} \| \hat{\zeta}^\theta_n-\zeta^\theta \|_{\mathcal{N}_{\Phi}(\Omega)}^{\frac{d}{2m}},\label{maximumineq}
	\end{eqnarray}
	Combining inequalities (\ref{l2basicineq})-(\ref{maximumineq}) yields
	\begin{eqnarray*}
		\begin{split}
			&\| \zeta^\theta-\hat{\zeta}^\theta_n\|_n^2+c_8\lambda_n \| \zeta^\theta-\hat{\zeta}^\theta_n\|^2_{\mathcal{N}_\Phi(\Omega)}\\
			&\leq c_9\lambda_n\bar{v} \|\zeta^\theta-\hat{\zeta}^\theta_n\|_n+2 t n^{-1/2}\| \hat{\zeta}^\theta_n-\zeta^\theta \|_n^{1-\frac{d}{2m}} \| \hat{\zeta}^\theta_n-\zeta^\theta \|_{\mathcal{N}_{\Phi}(\Omega)}^{\frac{d}{2m}}.
		\end{split}	
	\end{eqnarray*}
	Then we obtain the desired results by applying Lemma \ref{lemma:order}.
\end{proof}

\begin{proof}[Proof of Theorem \ref{th:calibration}]
	Under the condition (\ref{lambdabound}), it is not hard to verify that $\bar{v}\lambda_n$ and $\bar{v}\lambda^{1/2}_n$ are bounded by a multiple of $tn^{-1/2}\lambda_n^{-\frac{d}{4m}}$ and $tn^{-1/2}\lambda_n^{-\frac{2m+d}{4m}}$, respectively. Thus Theorem \ref{th:uniform} gives
	\begin{eqnarray}
	&&\sup_{\theta\in\Theta}\|\hat{\zeta}^\theta_n-\zeta^\theta\|_n\leq c_2 \bar{v}\lambda_n,\label{bp1}\\
	&&\sup_{\theta\in\Theta}\|\hat{\zeta}^\theta_n-\zeta^\theta\|_{\mathcal{N}_\Phi(\Omega)}\leq c_3 \bar{v}\lambda_n^{1/2}.\label{bp2}
	\end{eqnarray}
	
	Using the definition of $\hat{\theta}_n$, we obtain the basic inequality
	\begin{equation*}
	\begin{split}
	&\frac{1}{n}\sum_{i=1}^n(\zeta_i^{\hat{\theta}_n}-\hat{\zeta}_n^{\hat{\theta}_n}(x_i))^2+\lambda_n\|\hat{\zeta}_n^{\hat{\theta}_n}\|^2_{\mathcal{N}_\Phi(\Omega)}\\
	&\leq \frac{1}{n}\sum_{i=1}^n(\zeta_i^{\theta'}-\hat{\zeta}_n^{\theta'}(x_i))^2+\lambda_n\|\hat{\zeta}_n^{\theta'}\|^2_{\mathcal{N}_\Phi(\Omega)},
	\end{split}
	\end{equation*}
	which is equivalent to
	\begin{eqnarray}\label{basicinequcalibration}
	\begin{split}
	&\lambda_n\left\{\|\zeta^{\hat{\theta}_n}\|^2_{\mathcal{N}_\Phi(\Omega)}-\|\zeta^{\theta'}\|^2_{\mathcal{N}_\Phi(\Omega)}\right\}\\
	&\leq \left\{\|\hat{\zeta}_n^{\theta'}-\zeta^{\theta'}\|_n^2-\|\hat{\zeta}_n^{\hat{\theta}_n}-\zeta^{\hat{\theta}_n}\|_n^2\right\}\\
	&+2\left\{\langle e,\hat{\zeta}_n^{\hat{\theta}_n}-\zeta^{\hat{\theta}_n}\rangle_n- \langle e,\hat{\zeta}_n^{\theta'}-\zeta^{\theta'}\rangle_n\right\}\\
	&+\lambda_n\left\{\|\zeta^{\hat{\theta}_n}\|^2_{\mathcal{N}_\Phi(\Omega)}-\|\hat{\zeta}_n^{\hat{\theta}_n}\|^2_{\mathcal{N}_\Phi(\Omega)}-\|\zeta^{\theta'}\|^2_{\mathcal{N}_\Phi(\Omega)}+\|\hat{\zeta}_n^{\theta'}\|^2_{\mathcal{N}_\Phi(\Omega)}\right\}\\
	&=:D_1+2D_2+\lambda_n D_3.
	\end{split}
	\end{eqnarray}
	
	Now we bound $D_1,D_2$ and $D_3$ conditional on the event $A_t$. Using (\ref{bp1}), we have
	\begin{eqnarray}\label{D1}
	D_1\leq \sup_{\theta\in\Theta}\|\hat{\zeta}^\theta_n-\zeta^\theta\|^2_n\leq c_4 \bar{v}^2 \lambda_n^2.
	\end{eqnarray}
	By (\ref{bp1})-(\ref{bp2}) and the definition of $A_t$, we obtain
	\begin{eqnarray}\label{D2}
	D_2&\leq& 2 \sup_{\theta\in\Theta}|\langle e,\hat{\zeta}_n^{\theta}-\zeta^{\theta}\rangle_n|\nonumber\\
	&\leq&2\sup_{g}\frac{|\langle e,g\rangle_n|}{\|g\|_n^{1-\frac{d}{2m}}\|g\|_{\mathcal{N}_\Phi(\Omega)}^{\frac{d}{2m}}}\sup_{\theta\in\Theta}\|\hat{\zeta}_n^{\theta}-\zeta^{\theta}\|_n^{1-\frac{d}{2m}}\sup_{\theta\in\Theta}\|\hat{\zeta}_n^{\theta}-\zeta^{\theta}\|_{\mathcal{N}_\Phi(\Omega)}^{\frac{d}{2m}}\nonumber\\
	&\leq& c_5 tn^{-\frac{1}{2}} \bar{v}\lambda_n^{\frac{4m-d}{4m}}\nonumber\\
	&\leq&c_6 \bar{v}^2\lambda_n^2,
	\end{eqnarray}
	where the last inequality follows from condition (\ref{lambdabound}).
	For any $\theta\in\Theta$, we bound
	\begin{eqnarray*}
		&&\left|\|\hat{\zeta}_n^{\theta}\|^2_{\mathcal{N}_\Phi(\Omega)}-\|\zeta^{\theta}\|^2_{\mathcal{N}_\Phi(\Omega)}\right|\nonumber\\
		&=&\left|\|\hat{\zeta}_n^\theta-\zeta^\theta\|^2_{\mathcal{N}_\Phi(\Omega)}+2\langle\hat{\zeta}_n^\theta-\zeta^\theta,\zeta^\theta\rangle_{\mathcal{N}_\Phi(\Omega)}\right|\nonumber\\
		&\leq&\|\hat{\zeta}_n^\theta-\zeta^\theta\|^2_{\mathcal{N}_\Phi(\Omega)}+2\left|\langle\hat{\zeta}_n^\theta-\zeta^\theta,\zeta^\theta\rangle_{\mathcal{N}_\Phi(\Omega)}\right|\nonumber\\
		&=&\|\hat{\zeta}_n^\theta-\zeta^\theta\|^2_{\mathcal{N}_\Phi(\Omega)}+2\left|\langle\hat{\zeta}_n^\theta-\zeta^\theta,v_\theta\rangle_{L_2(\Omega)}\right|\nonumber\\
		&\leq&\|\hat{\zeta}_n^\theta-\zeta^\theta\|^2_{\mathcal{N}_\Phi(\Omega)}+2\|v_\theta\|_{L_2(\Omega)}\|\hat{\zeta}_n^\theta-\zeta^\theta\|_{L_2(\Omega)}\nonumber\\
		&\leq&c_3^2\bar{v}^2\lambda_n+2 c_2 \bar{v}^2 \lambda_n\nonumber\\
		&=&c_9 \bar{v}^2 \lambda_n,
	\end{eqnarray*}
	where the second equality follows from (\ref{normequality}); the second inequality follows from Cauchy-Schwarz inequality; the third inequality follows from (\ref{bp1}) and (\ref{bp2}). Therefore, we obtain the bound
	\begin{eqnarray}\label{D3}
	D_3\leq 2 c_9 \bar{v}^2\lambda_n.
	\end{eqnarray}
	Combining (\ref{basicinequcalibration}), (\ref{D1}), (\ref{D2}) and (\ref{D3}) and using the condition $t>1$ yields
	\begin{eqnarray}\label{final}
	\|\zeta^{\hat{\theta}_n}\|^2_{\mathcal{N}_\Phi(\Omega)}-\|\zeta^{\theta'}\|^2_{\mathcal{N}_\Phi(\Omega)}\leq c_{10} \bar{v}^2\lambda_n.
	\end{eqnarray}
	
	The assumption (\ref{curvature}) implies
	\begin{eqnarray*}
		a_2\min\{\|\hat{\theta}_n-\theta'\|^\gamma,a_3\}
		\leq \|\zeta^{\hat{\theta}_n}\|^2_{\mathcal{N}_\Phi(\Omega)}-\|\zeta^{\hat{\theta}_n}\|^2_{\mathcal{N}_\Phi(\Omega)},
	\end{eqnarray*}
	which, together with (\ref{final}) and the condition $\bar{v}^2\lambda_n<c_1:=a_3/(a_2 c_{10})$, yields the desired results.
\end{proof}

\bibliographystyle{siamplain}
\bibliography{all.ref}

\end{document}

%% file: tmp_improved_rate_header.tex
	\title{On the Improved Rates of Convergence for Mat\'ern-type Kernel Ridge Regression, with Application to Calibration of Computer Models\thanks{Submitted to the editors DATE.
			\funding{Tuo's work is supported by NSF grants DMS-1914636 and DMS-1564438, and also by the National Center for Mathematics and Interdisciplinary Sciences in CAS and NSFC grants 11501551, 11271355 and 11671386. Wu's work is supported by NSF grants DMS-1564438 and DMS-1914632.}}}

	\author{Rui Tuo\thanks{Department of Industrial and Systems Engineering, Texas A\&M University,	College Station, TX 77843, USA (\email{ruituo@tamu.edu}).}
		\and Yan Wang\thanks{College of Appied Sciences, Beijing University of Technology, Beijing 100124, China (\email{yanwang@bjut.edu.cn}).}
		\and C. F. Jeff Wu\thanks{School of Industrial and Systems Engineering,	Georgia Institute of Technology, Atlanta, Georgia 30332, USA (\email{jeff.wu@isye.gatech.edu})}}

	\headers{Improved Rates for Kernel Ridge Regression}{Rui Tuo, Yan Wang and C. F. Jeff Wu}